\documentclass[11pt]{article}
\usepackage{amssymb,bbm}    
\usepackage{amsmath}    
\usepackage{amsthm} 
\usepackage{tikz,color}  
\newtheorem{theorem}{Theorem}
\newtheorem{proposition}{Proposition}

\newtheorem{lemma}{Lemma}

\usepackage[comma,authoryear]{natbib}
\RequirePackage[colorlinks,citecolor=blue,urlcolor=blue,breaklinks]{hyperref}

\def\R{{\mathbb R}}
\def\Rd{{\mathbb R}^d}

\def\P{{\cal P}}

\def\PROB {{\mathbb P}}
\def\EXP {{\mathbb E}}
\def\IND{{\mathbb I}}
\def\Var{{\mathbb Var}}

\begin{document}
\begin{titlepage}
\thispagestyle{empty}
\setcounter{page}{0}

\title{Strongly universally consistent nonparametric regression and classification with privatised data}
\author{ Thomas Berrett\thanks{ Department of Statistics, University of Warwick, Coventry, CV4 7AL, United Kingdom. tom.berrett@warwick.ac.uk}
\and L\'aszl\'o Gy\"orfi\thanks{Department of Computer Science and Information Theory, Budapest University of Technology and Economics, Magyar Tud\'{o}sok krt. 2., Budapest, H-1117, Hungary.
\break
gyorfi@cs.bme.hu}
\and Harro Walk\thanks{Institute of Stochastics and Applications, University of Stuttgart, Pfaffenwaldring 57, D-70569 Stuttgart, Germany. harro.walk@mathematik.uni-stuttgart.de}}

\maketitle

\begin{abstract}
In this paper we revisit the classical problem of nonparametric regression, but impose local differential privacy constraints. Under such constraints, the raw data $(X_1,Y_1),\ldots,(X_n,Y_n)$, taking values in $\mathbb{R}^d \times \mathbb{R}$, cannot be directly observed, and all estimators are functions of the randomised output from a suitable privacy mechanism. The statistician is free to choose the form of the privacy mechanism, and here we add Laplace distributed noise to a discretisation of the location of a feature vector $X_i$ and to the value of its response variable $Y_i$.
Based on this randomised data, we design a novel estimator of the regression function, which can be viewed as a privatised version of the well-studied partitioning regression estimator.
The main result is that the estimator is strongly universally consistent.
Our methods and analysis also give rise to a strongly universally consistent binary classification rule for locally differentially private data.
\end{abstract}

\noindent

{\sc AMS Classification}: 62G08, 62G20.

\noindent

{\sc Key words and phrases}: regression estimate, classification, local differential privacy, universal consistency

\end{titlepage}

\section{Introduction}

In recent years there has been a surge of interest in data analysis methodology that is able to achieve strong statistical performance without comprimising the privacy and security of individual data holders. This has often been driven by applications in modern technology, for example by Google~\citep{erlingsson2014rappor}, Apple~\citep{tang2017privacy}, and Microsoft~\citep{ding2017collecting}, but the study goes at least as far back as~\citet{warner1965randomized} and is often used in more traditional fields of clinical trials~\citep{vu2009differential,dankar2013practicing} and census data~\citep{machanavajjhala2008privacy,dwork2019differential}. While there has long been an awareness that sensitive data must be anonymised, it has become apparent only relatively recently that simply removing names and addresses is insufficient in many cases \citep[e.g.][]{sweeney2002k,rocher2019estimating}. The concept of differential privacy \citep{dwork2006calibrating} was introduced to provide a rigorous notion of the amount of private information on individuals published statistics contain. Statistical treatments of this framework include~\citet{wasserman2010statistical,lei2011differentially,avella2019differentially,cai2019cost}

Although it is a suitable constraint for many problems, procedures that are differentially private often require the presence of a third party, who may be trusted to handle the raw data before statistics are published. To address this shortcoming, the local differential privacy constraint \citep[see, for example,][and the references therein]{kairouz2014extremal,duchi2018minimax} was introduced to provide a setting where analysis must be carried out in such a way that each raw data point is only ever seen by the original data holder. The simplest example of a locally differentially private mechanism is the randomised response~\citep{warner1965randomized} used with binary data, but mechanisms have also been developed for tasks such as classification~\citep{berrett2019classification}, generalised linear modelling~\citep{duchi2018minimax}, empirical risk minimisation~\citep{wang2018empirical}, density estimation~\citep{butucea2020local}, functional estimation~\citep{rohde2018geometrizing} and goodness-of-fit testing~\citep{berrett2020locally}.

Regression is a cornerstone of modern statistical analysis, routinely used across the sciences and beyond.
We recall that, in a standard stochastic model, a regression estimator  predicts for an observed $d$ dimensional random feature vector an unknown random response, with finite second moment. The regression function, given by the conditional expectation of the response given the feature vector, achieves minimum mean squared error. 
Typically, the statistician does not know the underlying stochastic structure, but has access to a corresponding finite sample of independent identically distributed design-response vectors in $\R^d \times \R$, and on this basis estimates the regression function. The background will be given below at the beginning of Section 2, and in the following we shall refer several times to the monograph of \citet{gyorfi2006distribution}.
A binary classification (pattern recognition) rule predicts for a feature vector an unknown random response taking values in $\{-1,1\}$. The so-called Bayes decision rule achieves minimum error probability (Bayes error). Given a finite sample of i.i.d. design-response vectors in  $\R^d \times \{-1, 1\}$, the  Bayes rule is approximated. We formulate the setup in  Section 4, while  the monograph of  \citet{devroye2013probabilistic} contains a detailed theory of nonparametric classification.

While regression has been relatively well-studied in the non-local model of differential privacy \citep[e.g.][]{cai2019cost}, results in the local model are scarce. \citet{zheng2017collect} studies sparse linear regression, kernel ridge regression and GLMs. \citet{smith2017interaction,wang2018empirical} study parametric empirical risk minimisation. \citet{wang2018high} studies sparse linear regression. \citet{duchi2018minimax,duchi2018right} study GLMs. The recent work~\citet{farokhi2020deconvoluting} concerns a relaxed version of the locally private regression model where responses can be observed exactly, and empirically studies a Nadaraya--Watson-type estimator, but we are unaware of any other work on locally private nonparametric regression. The simpler problem of binary classification is studied in~\citep{berrett2019classification}, but there are significant additional challenges in designing a suitable estimator for the regression problem.

In this paper we introduce and investigate a new method for nonparametric regression under $\alpha$-local differential privacy constraints and also present a corresponding classification rule.
For regression our procedure combines a simple non-interactive privacy mechanism with a cubic partitioning regression estimate modifying the regressogram, which was originally introduced by~\citet{tukey1947non} and has been well-studied since \citep[see, e.g.,][Chapter~4 and Section 23.1, and the references therein]{gyorfi2006distribution}. In Section 3 we describe the procedure and state that the sequence of estimates is strongly universally consistent, in that the $L_2$-risk converges almost surely to zero  in the large sample limit for any data-generating distribution for which the response has a finite second moment. Let us mention that in the degenerate case without privacy the estimator reduces to the strongly universally consistent partitioning estimator of \citet{gyorfi1991universal}.
The problem of classification is strictly easier than regression, therefore our methods and analysis also give rise to a strongly universally consistent binary classification rule for locally differentially private data.

The remainder of the paper is organised as follows. In Section~\ref{Sec:Prelim} we introduce the necessary background on regression and local differential privacy. In Section~\ref{Sec:MainRegression} we introduce our privacy mechanism and estimators, and state our main results in the regression setting. In Section~\ref{Sec:Classification} we study the consequences of the results in Section~\ref{Sec:MainRegression} for binary classification. All proofs will be deferred to Section~\ref{Sec:Proofs}. We intend to investigate the rates of convergence for the locally private regression problem in a subsequent paper.

\section{Preliminaries}
\label{Sec:Prelim}

\subsection{Background and non-private setting}
Let $(X,Y)$ be a pair of random variables such that the feature vector $X$ takes values in $\R^d$
and its response variable $Y$ is a real-valued random variable with $\EXP[Y^2]<\infty$.
We denote by $ \mu $ the distribution of the feature vector $X$, that is, for all measurable sets $A\subset \R^d$, we have
$\mu(A)=\PROB\{X\in A\}$.
Then the \emph{regression function}
\begin{equation}
\label{def}
 m(x) = \EXP [Y \mid X=x]
\end{equation}
is well defined for $\mu$-almost all $x$.
For each measurable function $g:\R^d\to\R$ one has
\[
\EXP \left[ \{ g(X) - Y\}^{2}\right] = \EXP  \left[\{ m(X) - Y \}^{2} \right]  + \EXP  \left[\{ m(X) - g(X) \}^{2} \right],
\]
therefore, with the notation
\[
L^*=\EXP  \left[\{ m(X) - Y \}^{2} \right]~,
\]
we have
\begin{equation}
\label{Eq:RiskDecomp}
\EXP \left[ \{ g(X) - Y\}^{2}\right] = L^*  + \int \{ m(x) - g(x) \}^{2} \mu(dx).
\end{equation}
We measure the performance of an estimator $\hat{m}$ of $m$ through the loss function
\[
	L(m, \hat{m}) := \int \{m(x) - \hat{m}(x)\}^2 \mu(dx),
\]
which, by~\eqref{Eq:RiskDecomp}, may be interpreted as the excess prediction risk for a new observation $X$.

In this paper we are mainly concerned with regression estimates $\hat{m}$ based on partitions of the sample space, which were originally studied by~\citet{tukey1947non}. Let $\P_h=\{A_{h,1},A_{h,2},\ldots\}$ be a cubic partition
of $\Rd$ such that the cells $A_{h,j}$ are cubes of volume
$h^d$.  If $x_{h,j}$ denotes the center of the cube $A_{h,j}$, then introduce the discretization of $x$ by the quantiser
\begin{align*}
Q_h(x):=x_{h,j}, \mbox{ if } x\in A_{h,j}.
\end{align*}
The raw data will be independent and identically distributed copies
\begin{align*}
\mathcal{D}_n := \{(X_1,Y_1),\ldots,(X_n,Y_n)\}
\end{align*}
of the random vector $(X,Y)$, and the estimators that we consider will be (randomised) functions of the binned data, defined by
\[
\{(Q_h(X_1),Y_1),\ldots,(Q_h(X_n),Y_n)\}.
\]
Using this binned data, when we do not have to satisfy privacy constraints, one may create a scheme for a public data set as follows: there are $n$ individuals in the study such that individual $i$ generates the sample pair $(X_i,Y_i)$ and he submits the discretised version $(Q_h(X_i),Y_i)$ to a data collector. The data collector calculates the empirical distributions
\begin{align*}
\nu_n(A_{h,j})=\frac 1n \sum_{i; Q_h(X_i)=x_{h,j}}Y_i
\end{align*}
and
\begin{align*}
\mu_n(A_{h,j})=\frac 1n \sum_{i; Q_h(X_i)=x_{h,j}} 1.
\end{align*}
Then, the public data set
\begin{align*}
D_{n,h}=\{( j, \nu_n(A_{h,j}),\mu_n(A_{h,j})); \mu_n(A_{h,j})>0\}
\end{align*}
is published. The data set $D_{n,h}$ has the favourable property that the size $\#(D_{n,h})$ is much less than $n$ \citep[cf.][]{lugosi1999adaptive}.

For binned data, the partitioning regression estimate is defined by
\[
m'_{n}(x)= \frac{\sum_{i=1}^n Y_i \IND_{\{X_i\in A_{h_n,j}\}} }
   {\sum_{i=1}^n \IND_{\{X_i\in A_{h_n,j}\}}}\qquad \mbox{if } x\in A_{h_n,j},
\]
where $0/0$ is $0$ by definition and $\IND$ denotes the indicator function.
In order to have strong universal consistency, we modify the partitioning regression estimate as follows:
\[
m_{n}(x)= \frac{ \nu_n(A_{h_n,j}) }{ \mu_n(A_{h_n,j})}\IND_{\{\mu_n(A_{h_n,j})\ge \log n/n\}}    \qquad \mbox{if } x\in A_{h_n,j}.
\]

\begin{theorem}
\label{bcons}
(Theorem 23.3 in  \cite{gyorfi2006distribution}.)
If
\[
\lim_{n\to\infty}{ h}_n=0\, \mbox{ and }\, \lim_{n\to\infty}n{ h}_n^{d}/\log n =\infty,
\]
then  the estimate $ m_{n}$ is strongly universally consistent, i.e.,
\begin{align}
\label{conv}
\lim_{n\to \infty} \int ( m(x) - m_n(x) )^{2} \mu(dx)=0
\end{align}
a.s. for any distribution of $(X,Y)$ with $\EXP Y^2<\infty$.
\end{theorem}

There is a huge literature on weak and strong universal consistency of regression estimates. Weak universal consistency means convergence
\begin{align*}
\lim_{n\to \infty} \int \EXP \left[ \{ m(x) - m_n(x) \}^{2} \right] \mu(dx)=0
\end{align*}
for any distribution of $(X,Y)$ with $\EXP Y^2<\infty$.
For the weak universal consistency of local averaging regression estimates $m_n$, which includes partitioning estimates, kernel estimates and nearest neighbor estimates,
we refer to Chapters 4 - 6 in \citet{gyorfi2006distribution}.

\subsection{Local differential privacy}

When working under privacy constraints, our estimator will not have direct access to the raw data $\mathcal{D}_n$, or even the binned data $\mathcal{D}_{n,h}$.
Instead, it will only be allowed to depend on randomised data $(Z_1,\ldots,Z_n)$, defined on some measurable space $(\mathcal{Z}^n, \mathcal{B}^n)$, that has been generated conditional on $\mathcal{D}_n$. A \emph{privacy mechanism} will be a conditional distribution $Q : \mathcal{B}^n \times (\mathbb{R}^d \times \mathbb{R})^{\otimes n} \rightarrow [0,1]$ with the interpretation that
\[
	(Z_1,\ldots,Z_n) | \{\mathcal{D}_n =\{(x_1,y_1),\ldots, (x_n,y_n)\} \} \sim Q( \cdot | (x_1,y_1), \ldots, (x_n,y_n) ).
\]
This privacy mechanism will be said to be \emph{sequentially interactive} \citep{duchi2018minimax} if it respects the graphical structure

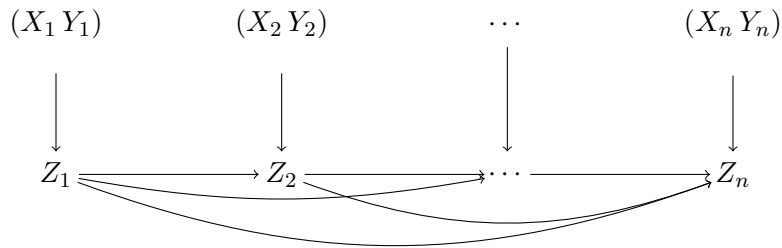
\begin{center}
\begin{tikzpicture}[->]
  \tikzstyle{vertex}=[circle,minimum size=17pt,inner sep=0pt]

\foreach \name/\text/\x in {X1/(X_1\,Y_1)/0, X2/(X_2\,Y_2)/3, Xdots/\cdots/6, Xn/(X_n\,Y_n)/9}
\node[vertex](\name) at (\x,0) {$\text$};

\foreach \name/\text/\x in {Z1/Z_1/0, Z2/Z_2/3, Zdots/\cdots/6, Zn/Z_n/9}
\node[vertex](\name) at (\x,-2) {$\text$};

\draw (X1) -- (Z1);
\draw (X2) -- (Z2);
\draw (Xdots) -- (Zdots);
\draw (Xn) -- (Zn);
\draw (Z1) -- (Z2);
\draw (Z1) to [out=-10, in=-170] (Zdots);
\draw (Z1) to [out=-20, in=-160] (Zn);
\draw (Z2) to (Zdots);
\draw (Z2) to [out=-20, in=-160] (Zn);
\draw (Zdots) -- (Zn);
\end{tikzpicture}
\end{center}

\noindent In particular, this requires that $Z_i \perp \!\!\! \perp (X_j,Y_j) | \{X_i,Y_i,Z_1,\ldots,Z_{i-1}\}$ for any $j \neq i$, so that $Z_i$ is generated with only the knowledge of $(X_i,Y_i)$ and $Z_1,\ldots,Z_{i-1}$. For this reason, such privacy mechanisms are locally private. Sequentially interactive privacy mechanisms may be specified by a sequence $(Q_1,\ldots,Q_n)$ of conditional distributions with $Q_i : \mathcal{B} \times (\mathbb{R}^d \times \mathbb{R}) \times \mathcal{Z}^{i-1} \rightarrow [0,1]$ and with the interpretation that
\[
	Z_i | \{(X_i,Y_i)\!=\!(x_i,y_i), Z_1=z_1, \ldots, Z_{i-1}=z_{i-1}\} \sim Q_i( \cdot | (x_i,y_i),z_1, \ldots,z_{i-1}).
\]
Given $\alpha>0$, a sequentially interactive mechanism specified by $(Q_1,\ldots,Q_n)$ will be said to be $\alpha$-\emph{locally differentially private} ($\alpha$-LDP) if
\[
	\sup_{A \in \mathcal{B}} \sup_{z_1,\ldots,z_{i-1} \in \mathcal{Z}} \sup_{(x_i,y_i),(x_i',y_i') \in \mathbb{R}^d \times \mathbb{R}} \frac{Q_i(A | (x_i,y_i), z_1,\ldots,z_{i-1})}{Q_i(A | (x_i',y_i'), z_1,\ldots,z_{i-1})} \leq e^\alpha
\]
for each $i=1,\ldots,n$. Let $\mathcal{Q}_\alpha$ denote the set of all $\alpha$-LDP privacy mechanisms. In the remainder of this paper, we will consider estimators $\hat{m}_n$ that are measurable functions of some $(Z_1,\ldots,Z_n)$ that has been generated by an $\alpha$-LDP privacy mechanism applied to the raw data $\mathcal{D}_n$.

The privacy mechanisms that we will consider here are actually of a simple, \emph{non-interactive} form, whereby we also have
\[
	Z_i \perp \!\!\!  \perp (X_j,Y_j,Z_j)
\]
for all $j \neq i$. In this case we have
\[
	Q(A_1,\ldots,A_n | (x_1, y_1), \ldots, (x_n,y_n)  ) = \prod_{i=1}^n Q_i(A_i | (x_i,y_i))
\]
for all $(A_1,\ldots,A_n) \in \mathcal{B}^n$. Such mechanisms satisfy the $\alpha$-LDP constraint if and only if
\[
	\sup_{A \in \mathcal{B}} \sup_{(x_i,y_i), (x_i',y_i') \in \mathbb{R}^d d\times \mathbb{R}} \frac{Q_i(A | x_i,y_i)}{Q_i(A | x_i',y_i')} \leq e^\alpha
\]
for each $i=1,\ldots,n$. Non-interactive mechanisms are computationally attractive in practice as they require minimal communication between the statistician and the orginal data holders, and in large-scale applications there are many practical barriers to interactivity \citep{joseph2019role}. In fact, we will see that in the problem considered here we are able to construct simple, non-interactive privacy mechanisms that give rise to intepretable, strongly universally consistent estimators.

\section{Our regression estimation method and its strong universal consistency}
\label{Sec:MainRegression}

Similarly to \cite{berrett2019classification} we consider locally privatised data given as follows:
the privacy mechanism is formulated by independent double arrays $\{\epsilon_{i,j}\}$ and $\{\zeta_{i,j}\}$ such that the elements of the arrays are i.i.d. with centred, unit-variance Laplace distributions.
For $i=1,\ldots,n$ and for $0<M_n\le \infty$,
write $[Y_i]_{-M_n}^{M_n} = \min\{M_n, \max(Y_i,-M_n) \}$ for the truncated response; it will be sometimes be convenient to write $[Y_i]_{-\infty}^{\infty} =Y$ for no truncation.
Choose a sphere $S_n$ centered at the origin.
Assume that the cells $A_{h,j}$ are numbered such that $A_{h,j}\cap S_n\ne \emptyset$ when $j\le N_n$ for some integer $N_n>0$, and $A_{h,j}\cap S_n= \emptyset$ otherwise.
Individual $i\le n$ generates and transmits the data
\begin{equation}
\label{Eq:Mech1}
	Z_{i,j}:= [Y_i]_{-M_n}^{M_n} \IND_{\{X_i \in A_{h,j}\}} + \sigma_Z \epsilon_{i,j}, \quad j\le N_n	
\end{equation}
and
\begin{equation}
\label{Eq:Mech2}
W_{i,j}:=\IND_{\{X_i\in A_{h,j} \}} + \sigma_W \zeta_{i,j}, \quad j\le N_n,
\end{equation}
where $\sigma_Z>0$ and $\sigma_W>0$.
This means that individual $i$ generates noisy data for any cell $A_{h,j}$ with $j\le N_n$. The following result studies the local differential privacy of this mechanism in the case that $N_n=\infty$, but it is a straightforward consequence of this that the mechanism satisfies the same bound when $N_n <\infty$.

\begin{proposition}
\label{Prop:BoundedLDP}
Consider the privacy mechanism defined in~\eqref{Eq:Mech1} and~\eqref{Eq:Mech2} when $\varepsilon_{1,1}$ and $\zeta_{1,1}$ have unit-variance Laplace distribution with probability density $x \mapsto \exp(-\sqrt{2}|x|) /\sqrt{2}$. Writing $q_{W,Z | X,Y}(w,z | x,y)$ for the probability density function of $( (W_{1,j})_{j=1}^\infty, (Z_{1,j})_{j=1}^\infty )$ conditional on $X_1=x, Y_1=y$, we have
\[
	\sup_{w,z \in \mathbb{R}^\mathbb{N}} \sup_{ x,x' \in \mathbb{R}^d} \sup_{y,y' \in [-M,M]} \frac{q_{W,Z|X,Y}(w,z|x,y)}{q_{W,Z|X,Y}(w,z|x',y')} \leq \exp \Bigl( 2^{3/2}/ \sigma_W + 2^{3/2}M / \sigma_Z \Bigr).
\]
\end{proposition}
Given $\alpha>0$, we can therefore ensure that our privacy mechanism is $\alpha$-LDP by choosing $M,\sigma_W,\sigma_Z$ such that $2^{3/2}(1/\sigma_W + M/\sigma_Z) \leq \alpha$. This is satisfied if, for example, we take $\sigma_W^2 = 32/\alpha^2$ and $\sigma_Z^2 = 32 M^2 / \alpha^2$. For such $\sigma_W,\sigma_Z$, the data set
\begin{align*}
\tilde D_{n,h}=\{( j, \tilde \nu_n(A_{h,j}),\tilde \mu_n(A_{h,j})) : j=1,\ldots,N_n \}
\end{align*}
may be published without violating the $\alpha$-LDP constraint,
where
\begin{equation}
\label{Eq:nummum}
	\tilde{\nu}_n(A_{h,j}) = \frac{1}{n} \sum_{i=1}^n Z_{i,j}\IND_{\{j\le N_n\}} \quad \text{and} \quad \tilde{\mu}_n(A_{h,j}) = \frac{1}{n} \sum_{i=1}^n W_{i,j}\IND_{\{j\le N_n\}}.
\end{equation}

Now that we have introduced our privacy mechanism we may define our estimator of $m$ based on $\tilde{D}_{n,h}$. For $c_n>0$ we define
\begin{align*}
	\tilde{m}_n(x) = \frac{\tilde{\nu}_n(A_{h_n,j})}{\tilde{\mu}_n(A_{h_n,j})}\IND_{\{\tilde\mu_n(A_{h_n,j})\ge c_nh_n^d\}}\IND_{\{j\le N_n\}} \quad \text{when } x \in A_{h_n,j}.
\end{align*}
This is a novel estimator that extends the classical partitioning regression estimate to the LDP setting. In non-private settings such estimators may be seen as averaging the value of the response over each element of the partition, but here we are unable to retain this interpretation as we cannot know exactly how many data points fall in each cell. This lack of knowledge is particularly problematic in low-density regions, where the estimate of $\mu$ is necessarily especially noisy, and where our estimator must be carefully defined. A crucial component of the estimate is the way it detects the empty cells and truncates. If $X$ has a density, then $\mu(A_{h_n,j})$ is of order $h_n^d$. Furthermore, on the support of an arbitrary $\mu$, $\mu(A_{h_n,j})/h_n^d$ is bounded away from zero. More precisely, if  $A_n(x)$ stands for the cube $A_{h_n,j}$ containing $x$, then
\begin{align*}
\liminf_n \mu(A_n(x))/h_n^d>0
\end{align*}
for $\mu$-almost all $x$, see Lemma 24.10 in  \cite{gyorfi2006distribution}. Thus, for arbitrary $\mu$, the order of $\mu(A_{h_n,j})$ is at least $h_n^d$. Therefore, $c_n \rightarrow 0$ implies that
$\mu(A_{h_n,j})>c_nh_n^d$, for large enough $n$.

When $\sigma_W=\sigma_Z=0$ and $c_n=\log n /(n h_n^d)$ then we recover the non-private partitioning estimator, which has access to the raw data, discussed above.

Our first main new result extends Theorem~\ref{bcons} to the private setting where $\sigma_W,\sigma_Z >0$ are fixed, and establishes the strong universal consistency of $\tilde{m}_n$.
\begin{theorem}
\label{Thm:UniversalConsistency}
If $S_n\uparrow \Rd$, $c_n \rightarrow 0$, $h_n \rightarrow 0$, $M_n\rightarrow \infty$ and
\begin{align}
\label{2d}
\frac{(\log n)^3  }{n c_n^2 h_n^{2d}} \rightarrow 0
\end{align}
then
\begin{align}
\label{str}
\lim_{n\to \infty} \int \{ m(x) - \tilde m_n(x) \}^{2} \mu(dx)=0\quad \mbox{a.s.,}
\end{align}
for any distribution of $(X,Y)$ with $\EXP Y^2<\infty$.
\end{theorem}

The proof of Theorem~\ref{Thm:UniversalConsistency} shows that replacement of (\ref{2d}) by $n c_n^2 h_n^{2d} \rightarrow \infty$
yields the weak universal consistency of $\tilde{m}_n$.

In problems of differential privacy one often wants to work in a high-privacy regime, where we have $\alpha \rightarrow 0$ as $n \rightarrow \infty$. With our privacy mechanism, this requires that $\min(\sigma_W,\sigma_Z/M) \rightarrow \infty$, and so we remark that Theorem~\ref{Thm:UniversalConsistency} can easily be extended to the setting in which the variances $\sigma^2_Z$ and $\sigma^2_W$ may depend on the sample size $n$.
Replacing the condition  (\ref{2d}) with
\begin{align*}
\frac{(\log n)^3 (1+\sigma^2_{Z,n} +\sigma^2_{W,n}) }{n c_n^2 h_n^{2d}}\rightarrow 0,
\end{align*}
a straightforward extension of the proof of Theorem~\ref{Thm:UniversalConsistency} implies the strong universal consistency.

Comparing with Theorem~\ref{bcons}, we see that that the usual condition $n h_n^d \rightarrow \infty$ has been replaced by $nh_n^{2d} \rightarrow \infty$.
Heuristically, this difference can be understood by considering the properties of $\tilde{\nu}_n(A_{h_n,j})$. Writing $\nu(A):= \int_A m(x) \mu(dx)$, we have
\[
	\mathbb{E} \{ \tilde{\nu}_n(A_{h_n,j}) \} = \nu(A_{h_n,j}),
\]
which is the same as in the non-private case. However, we see a difference when we consider that
\begin{equation}
\label{Eq:HomoscedasticVariance}
	n \Var \{\tilde{\nu}_n(A_{h_n,j})\} = n \Var \{\nu_n(A_{h_n,j})\} + \sigma_Z^2.
\end{equation}
In the non-private case, the only contribution is from the first term, which can be seen to typically be $O(h_n^d)$. However, in the private case we will usually take $\sigma_Z \propto 1/\alpha$ to be large, and hence the variance in~\eqref{Eq:HomoscedasticVariance} is dominated by the second term, which does not vanish with $h_n$. This occurs in other LDP problems \citep[e.g.][]{berrett2020locally}; the privacy constraint introduces an unavoidable homoscedastic term into the variance of our estimator, which results in very different behaviour, including a curse-of-dimensionality that is often more severe than in non-private problems.

Assume that the regression function $m$ is Lipschitz continuous, $Y$ is bounded and $X$ has a density, which is bounded away from zero. Then following the line of the proof of Theorem~\ref{Thm:UniversalConsistency} we can bound the rate of convergence:
\begin{align*}
\EXP\int \{m(x)-\tilde{m}_n(x)\}^2\mu(dx)
&=
O\left(\frac{ 1 }{n c_n^2h_n^{2d}}\right)+O(h_n^2).
\end{align*}
For the choices
\[
{ h}_n=
c'
n^{-1/(2(d+1))}
\]
and
\[
c_n=
1/\sqrt{\log n},
\]
this upper bound results in
\begin{align*}
\EXP \int \{ m(x) - \tilde{m}_n(x) \}^{2} \mu(dx)
&=
O\left(\frac{\log n}{n^{1 / (d+1)}}\right).
\end{align*}
We conjecture that
\begin{align*}
O\left(\frac{1}{n^{1 / (d+1)}}\right)
\end{align*}
is the minimax lower bound over all $\alpha$-LDP privacy mechanisms for Lipschitz continuous regression function, which would imply that our estimate is minimax optimal up to a factor of $\log n$.
Furthermore, the lower bound on the density appears to be crucial; we speculate that if the density is not bounded away from zero, then the rate of convergence of any estimate can be arbitrarily slow.

\section{Consequences in classification}
\label{Sec:Classification}

For the setup of binary classification, let the feature vector $X$ take values in $\Rd$, and let its label
$Y$ be $\pm 1$ valued. If $g$ is an arbitrary decision function
then its error probability is denoted by
\[
L(g)=\PROB\{g(X)\ne Y\}.
\]
The Bayes decision rule $g^*$, given by 
\[
g^*(x) = sign\,  m(x),
\]
where $sign(z)=1$ for $z>0$ and $sign(z)=-1$ for $z \leq 0$, minimises the error probability. Let
\[
L^*=\PROB\{g^*(X)\ne Y\}
\]
denotes its error probability.

For privatised data, the partitioning classification rule is defined by
\begin{align*}
g_n(x) = sign\,\left(  \tilde{\nu}_n(A_{h_n,j})\right) \quad \text{when } x \in A_{h_n,j}.
\end{align*}
Note that this rule does not use the data $\{W_{i,j}\}$. Under the conditions
\begin{align*}
\lim_{n\to\infty}{ h}_n=0\, \mbox{ and }\, \lim_{n\to\infty}n{ h}_n^{2d} =\infty,
\end{align*}
\citet{berrett2019classification} showed that  the partitioning classification rule $g_n$ is weakly universally consistent, i.e.,
\begin{align*}
\lim_{n\to\infty} \EXP\{L(g_n)\}=L^*
\end{align*}
for any distribution of $(X,Y)$. Our work here allows us to strengthen this result to the following theorem on strong universal consistency:
\begin{theorem}
\label{patt}
If
\begin{align*}
\lim_{n\to\infty}{ h}_n=0\, \mbox{ and }\, \lim_{n\to\infty}n{ h}_n^{2d}/\log n =\infty,
\end{align*}
then the classification rule $ g_{n}$ is strongly universally consistent, i.e.,
\begin{align*}
\lim_{n\to \infty} L(g_n)=L^*
\end{align*}
a.s. for any distribution of $(X,Y)$.
\end{theorem}

The rates of convergence of the classification rule $g_n$, over classes of data-generating mechanisms satisying H\"older continuity and a strong density assumption, were established in \citet{berrett2019classification}, and were moreover shown to match a minimax lower bound. In future work we will aim to establish rates of convergence for the regression problem.

\section{Proofs and auxiliary results}
\label{Sec:Proofs}

\begin{proof}[Proof of Proposition~\ref{Prop:BoundedLDP}]
Fix any $w,z \in \mathbb{R}^\mathbb{N}, x,x' \in \mathbb{R}^d, y,y' \in [-M,M]$. We have
\begin{align*}
	& \frac{f_{W,Z|X,Y}(w,z|x,y)}{ f_{W,Z|X,Y}(w,z|x',y')} \\
	& = \exp \Bigl( (\sqrt{2}/\sigma_W) \sum_{j=1}^\infty (|w_j- \mathbbm{1}_{\{ x' \in A_{h,j}\}}| - |w_j- \mathbbm{1}_{\{ x \in A_{h,j}\}}| ) \\
	& \hspace{50pt} + (\sqrt{2}/\sigma_Z) \sum_{j=1}^\infty ( |z_j - y' \mathbbm{1}_{\{x' \in A_{h,j}\}} | - |z_j - y \mathbbm{1}_{\{x \in A_{h,j}\}} |) \Bigr).
\end{align*}
Now, if there exists $j \in \mathbb{N}$ such that $x,x' \in A_{h,j}$, we have
\begin{align*}
	\sum_{j'=1}^\infty (|w_{j'}- \mathbbm{1}_{\{ x' \in A_{h,j'}\}}| - |w_{j'}- \mathbbm{1}_{\{ x \in A_{h,j'}\}}| ) &= 0 \\
	\sum_{j'=1}^\infty ( |z_{j'} - y' \mathbbm{1}_{\{x' \in A_{h,j'}\}} | - |z_{j'} - y \mathbbm{1}_{\{x \in A_{h,j'}\}} |) &= |z_j - y'| - |z_j - y| \leq 2M.
\end{align*}
On the other hand, if $x \in A_{h,j}$ and $x' \in A_{h,j'}$ with $j \neq j'$, then we have
\begin{align*}
&\sum_{j''=1}^\infty (|w_{j''}- \mathbbm{1}_{\{ x' \in A_{h,j''}\}}| - |w_{j''}- \mathbbm{1}_{\{ x \in A_{h,j''}\}}| )\\
&\quad = |w_j| - |w_j -1| + |w_{j'}-1| - |w_{j'}| \leq 2 ,\\
&\sum_{j''=1}^\infty (|z_{j''}- y'\mathbbm{1}_{\{ x' \in A_{h,j''}\}}| - |z_{j''}- y \mathbbm{1}_{\{ x \in A_{h,j''}\}}| )\\
&\quad = |z_j| - |z_j - y| + |z_{j'} - y'| - |z_{j'}| \leq 2M.
\end{align*}
It therefore follows that
\[
	\frac{f_{W,Z|X,Y}(w,z|x,y)}{ f_{W,Z|X,Y}(w,z|x',y')} \leq \exp \Bigl( 2^{3/2}/ \sigma_W + 2^{3/2}M / \sigma_Z \Bigr),
\]
as required.
\end{proof}

The proof of Theorem~\ref{Thm:UniversalConsistency} uses two lemmas.

\begin{lemma}
\label{A}
For $0<\varepsilon<2$ and $\zeta_1,\ldots,\zeta_n$ i.i.d. with mean-zero, unit-variance Laplace distribution, one has
\begin{align*}
\PROB\left\{\left|\frac 1n \sum_{i=1}^n\zeta_i\right|\ge \varepsilon \right\}
&\le
2e^{-n\varepsilon^2/4}.
\end{align*}
\end{lemma}

\begin{proof}
Taking $t=n\varepsilon/2$ and using the fact that $\log(1-x) \geq -2x$ for $x \in [0,1/2]$, we have
\begin{align*}
\mathbb{P} \Bigl(  n^{-1} \sum_{i=1}^n \zeta_i \geq \varepsilon \Bigr)
&\leq e^{-t\varepsilon} \mathbb{E} [ \exp( t \zeta_1 /n) ]^n \\
& = \exp \Bigl( - t\varepsilon - n \log \Bigl( 1- \frac{t^2}{2n^2} \Bigr) \Bigr)\\
& \leq \exp\Bigl( - t\varepsilon + \frac{t^2 }{n} \Bigr) \\
& = e^{-n\varepsilon^2/4}.
\end{align*}
An analogous bound holds for the lower tail of the distribution, and the result follows.
\end{proof}

\begin{lemma}
\label{B}
Let $Z=(Z_1,\ldots,Z_n)$ be a collection of i.i.d. random
variables taking values in some measurable set $A$.
Let $f:A^n\to \R$ be  a measurable, symmetric, real-valued function, such that $f(Z_1,\dots ,Z_n)$ is integrable, let $g:A^{n-1}\to \R$ be the function obtained from $f$ by dropping the first argument.
Then for any integer $q\ge 1$,
\begin{align*}
&\EXP \left[(f(Z_1,\ldots,Z_n)-\EXP f(Z_1,\ldots,Z_n))^{2q} \right]\\
&\le 2(c^*q)^qn^q\EXP \left[  \left(f(Z_1,\ldots,Z_n)-g(Z_2,\ldots,Z_n)\right)^{2q}\right]~.
\end{align*}
with a universal constant $c^*<5.1$.
\end{lemma}

\begin{proof}
Applying Jensen's inequality, this lemma is a special case of Lemma 4.4 in \citet{devroye2018nearest}.
\end{proof}

\begin{proof}[Proof of Theorem~\ref{Thm:UniversalConsistency}]
We use the decomposition
\begin{align*}
\tilde m_n=m'_n+m^*_n,
\end{align*}
where for $x \in A_{h_n,j}$ we write
\begin{align*}
m'_n(x) = \frac{\frac{\sigma_Z}{n} \sum_{i=1}^n\epsilon_{i,j}}{ \tilde{\mu}_n(A_{h_n,j})}\IND_{\{\tilde\mu_n(A_{h_n,j})\ge c_nh_n^d\}}\IND_{\{j\le N_n\}},
\end{align*}
and
\begin{align*}
m^*_n(x) = \frac{\nu_n(A_{h_n,j})}{ \tilde{\mu}_n(A_{h_n,j})}\IND_{\{\tilde\mu_n(A_{h_n,j})\ge c_nh_n^d\}}\IND_{\{j\le N_n\}}.
\end{align*}
It suffices to show that
\begin{align}
\label{1*}
\lim_{n\to \infty} \int  m'_n(x)^{2} \mu(dx)=0\quad \mbox{a.s.,}
\end{align}
and
\begin{align}
\label{1**}
\lim_{n\to \infty} \int \{ m(x) -  m^*_n(x) \}^2 \mu(dx)=0\quad \mbox{a.s.}
\end{align}
But (\ref{conv}) implies that
\begin{align}
\label{c1}
\lim_{n\to \infty} \int \{ m(x) -  m_n(x) \}^2 \mu(dx)=0\quad \mbox{a.s.,}
\end{align}
for any distribution of $(X,Y)$ with $\EXP (Y^2)<\infty$,
and in order to prove (\ref{1**}) it therefore suffices to show that
\begin{align}
\label{c2}
\lim_{n\to \infty} \int \{ m_n(x) - m^*_n(x) \}^2 \mu(dx)=0\quad \mbox{a.s.,}
\end{align}
for any distribution of $(X,Y)$ with $\EXP (Y^2)<\infty$.

\bigskip
\noindent
{\it Proof of (\ref{1*}).}
Because of
\begin{align*}
\int  m'_n(x)^{2} \mu(dx)
&=
\sum_j\frac{\left(\frac{\sigma_Z}{n} \sum_{i=1}^n\epsilon_{i,j}\right)^2}{ \tilde{\mu}_n(A_{h_n,j})^2}\IND_{\{\tilde\mu_n(A_{h_n,j})\ge c_nh_n^d\}} \IND_{\{j \leq N_n\}} \mu(A_{h_n,j})\\
&\le \sigma_Z^2
\sum_j\frac{\left(\frac{1}{n} \sum_{i=1}^n\epsilon_{i,j}\right)^2}{ c^2_nh_n^{2d}}\mu(A_{h_n,j}),
\end{align*}
it suffices to show that
\begin{align}
\label{11}
\lim_{n\to \infty} \sigma_Z^2 \sum_j\frac{\left(\frac{1}{n} \sum_{i=1}^n\epsilon_{i,j}\right)^2}{ c^2_nh_n^{2d}}\mu(A_{h_n,j})=0\quad \mbox{a.s.}
\end{align}
We note
\begin{align*}
\EXP\{\epsilon_{1,1}^{2q}\}
&=2^{-q}(2q)!
\le 2^{-q}(2q)^{2q}e^{-2q/3}
= 2^{q}q^{2q}e^{-2q/3},
\end{align*}
which together with Lemma \ref{B} implies
\begin{align*}
\EXP\left\{\left(\sum_{i=1}^n\epsilon_{i,1}\right)^{2q}\right\}
&\le 2(c^*q)^qn^q\EXP\{\epsilon_{1,1}^{2q}\}
\le  2^{q+1}{c^*}^qq^{3q}e^{-2q/3}n^q .
\end{align*}
Jensen's inequality yields
\begin{align*}
&\PROB\left\{ \sum_j\frac{\left(\frac{1}{n} \sum_{i=1}^n\epsilon_{i,j}\right)^2}{ c^2_nh_n^{2d}}\mu(A_{h_n,j})>\varepsilon \right\}\\
&=
\PROB\left\{ \left(\sum_j\frac{\left(\frac{1}{n} \sum_{i=1}^n\epsilon_{i,j}\right)^2}{ c^2_nh_n^{2d}}\mu(A_{h_n,j})\right)^q>\varepsilon^q \right\}\\
&\le
\PROB\left\{ \sum_j\frac{\left(\frac{1}{n} \sum_{i=1}^n\epsilon_{i,j}\right)^{2q}}{ c^{2q}_nh_n^{2qd}}\mu(A_{h_n,j})>\varepsilon^q \right\}\\
&\le
\varepsilon^{-q} \frac{\EXP\left\{\left(\frac{1}{n} \sum_{i=1}^n\epsilon_{i,1}\right)^{2q}\right\}}{ c^{2q}_nh_n^{2qd}}.
\end{align*}
Thus,
\begin{align*}
\PROB\left\{ \sum_j\frac{\left(\frac{1}{n} \sum_{i=1}^n\epsilon_{i,j}\right)^2}{ c^2_nh_n^{2d}}\mu(A_{h_n,j})>\varepsilon \right\}
&\le
 \frac{\varepsilon^{-q}}{ c^{2q}_nh_n^{2qd}n^{2q}}\EXP\left\{\left( \sum_{i=1}^n\epsilon_{i,1}\right)^{2q}\right\}\\
&\le
2 \frac{\varepsilon^{-q}2^{q}{c^*}^qq^{3q}e^{-2q/3}}{ c^{2q}_nh_n^{2qd}n^{q}}\\
&=
2 \left(\frac{q^{3}}{ nc^{2}_nh_n^{2d}\varepsilon/(2c^*e^{-2/3})}\right)^q .
\end{align*}
Choose
\begin{align*}
q:= \lfloor (nc^{2}_nh_n^{2d}\varepsilon/(2c^*e^{1/3}))^{1/3}\rfloor .
\end{align*}
Then
\begin{align*}
\PROB\left\{ \sum_j\frac{\left(\frac{1}{n} \sum_{i=1}^n\epsilon_{i,j}\right)^2}{ c^2_nh_n^{2d}}\mu(A_{h_n,j})>\varepsilon \right\}
&\le
2e^{-(nc^{2}_nh_n^{2d}\varepsilon)^{1/3}/(2c^*e^{1/3})^{1/3}+1}.
\end{align*}
Condition (\ref{2d}) yields
\begin{align*}
\sum_n\PROB\left\{ \sum_j\frac{\left(\frac{1}{n} \sum_{i=1}^n\epsilon_{i,j}\right)^2}{ c^2_nh_n^{2d}}\mu(A_{h_n,j})>\varepsilon \right\}
&<
\infty
\end{align*}
and thus the Borel-Cantelli lemma results in  (\ref{11}).

\bigskip
\noindent
{\it Proof of (\ref{c2}).}
If in the definition of $m_n$ we modify $\nu_n$ such that
\[
\nu_n(A_{h,j}) = \frac{1}{n} \sum_{i=1}^n [Y_i]_{-M_n}^{M_n} \IND_{\{X_i \in A_{h,j}\}},
\]
then a slight modification of the proof of Theorem 23.3 in \citet{gyorfi2006distribution}
together with the condition $M_n\to \infty$ implies (\ref{conv}), too.
We have that
\begin{align*}
&\int \{ m_n(x) -  m^*_n(x) \}^2 \mu(dx)\\
&=
\sum_{j=1}^{N_n} \left\{ \frac{\nu_n(A_{h_n,j})}{ \mu_n(A_{h_n,j})}\IND_{\{\mu_n(A_{h_n,j})\ge \log n/n\}} - \frac{\nu_n(A_{h_n,j})}{ \tilde{\mu}_n(A_{h_n,j})}\IND_{D_n(A_{h_n,j})}   \right\}^2  \mu(A_{h_n,j} )\\
&+
\sum_{j=N_n+1}^{\infty} \left\{ \frac{\nu_n(A_{h_n,j})}{ \mu_n(A_{h_n,j})}\IND_{\{\mu_n(A_{h_n,j})\ge \log n/n\}}   \right\}^2\mu(A_{h_n,j} )\\
&\le 
\sum_{j} \left\{\frac{\nu_n(A_{h_n,j})}{ \mu_n(A_{h_n,j})}\IND_{\{\mu_n(A_{h_n,j})\ge \log n/n\}} - \frac{\nu_n(A_{h_n,j})}{ \tilde{\mu}_n(A_{h_n,j})}\IND_{D_n(A_{h_n,j})}   \right\}^2\mu(A_{h_n,j} )\\
&+
\int_{S_n^c}m_n(x)^2\mu(dx),
\end{align*}
where
\begin{align*}
D_n(A_{h_n,j})
&= \left\{\tilde\mu_n(A_{h_n,j})\ge c_nh_n^d \right\}
= \left\{\mu_n(A_{h_n,j})+\tau_n(A_{h_n,j})\ge c_nh_n^d \right\}
\end{align*}
with $\tau_n(A_{h_n,j})=\frac{\sigma_W}{n} \sum_{i=1}^n\zeta_{i,j}$.
Since we have $\int m(x)^2 \mu(dx)<\infty$, then (\ref{conv}) together with $S_n\uparrow \Rd$ yields that
\begin{align*}
\int_{S_n^c}m_n(x)^2 \mu(dx)
&\to 0
\end{align*}
a.s.
Now (\ref{2d}) implies that
\begin{align*}
c_nh_n^d
\ge \log n/n
\end{align*}
if $n$ is large enough. For such large $n$, set
\begin{align*}
E_n
&:=
\sum_{j} \left\{ \frac{\nu_n(A_{h_n,j})}{ \mu_n(A_{h_n,j})}\IND_{\{\mu_n(A_{h_n,j})\ge \log n/n\}} - \frac{\nu_n(A_{h_n,j})}{ \tilde{\mu}_n(A_{h_n,j})}\IND_{D_n(A_{h_n,j})}   \right\}^2 \mu(A_{h_n,j} )\\
&\hspace{-6pt} =
\sum_j \frac{\nu_n(A_{h_n,j})^2}{ \mu_n(A_{h_n,j})^2}\IND_{\{\mu_n(A_{h_n,j})\ge \log n/n\}}\left\{1 - \frac{\mu_n(A_{h_n,j})}{ \tilde{\mu}_n(A_{h_n,j})}\IND_{D_n(A_{h_n,j})}   \right\}^2 \mu(A_{h_n,j} ).
\end{align*}
Note that
\begin{align*}
\left|1 - \frac{\mu_n(A_{h_n,j})}{ \tilde{\mu}_n(A_{h_n,j})}\IND_{D_n(A_{h_n,j})}   \right|
&=
\left|1 - \frac{\mu_n(A_{h_n,j})}{ \tilde{\mu}_n(A_{h_n,j})}\right|\IND_{D_n(A_{h_n,j})}+\IND_{D_n(A_{h_n,j})^c}\\
&=
\frac{|\tau_n(A_{h_n,j})|}{\tilde{\mu}_n(A_{h_n,j})}\IND_{D_n(A_{h_n,j})}+\IND_{D_n(A_{h_n,j})^c}\\
&=
\frac{|\tau_n(A_{h_n,j})|}{\tilde{\mu}_n(A_{h_n,j})}\IND_{\{\tilde{\mu}_n(A_{h_n,j})\ge c_nh_n^d\}}+\IND_{\{\tilde{\mu}_n(A_{h_n,j})< c_nh_n^d\}}.
\end{align*}
Let $A_n(x)$ denote the cube $A_{h_n,j}$, which contains $x$.
Then,
\begin{align*}
E_n
& =
\int m_n(x)^2\frac{\tau_n(A_n(x))^2}{\tilde{\mu}_n(A_n(x))^2}\IND_{D_n(A_n(x))}\mu(dx )\\
& \hspace{100pt}+
\int m_n(x)^2\IND_{\{\tilde{\mu}_n(A_n(x))< c_nh_n^d\}}\mu(dx )\\
&\le
\int m_n(x)^2 \frac{\tau_n(A_n(x))^2}{\tilde{\mu}_n(A_n(x))^2}\IND_{D_n(A_n(x))}\mu(dx )\\
&\hspace{100pt} +2
\int \{m_n(x)-m(x)\}^2\mu(dx )\\
&\hspace{100pt}+
2 \int m(x)^2 \IND_{\{\tilde{\mu}_n(A_n(x))< c_nh_n^d\}}\mu(dx )\\
&=:
F_n+G_n+H_n.
\end{align*}
Define the notation
\begin{align*}
\mu^*(A):=\int_A m(x)^2 \mu(dx)
\end{align*}
and
\begin{align*}
\mu_n^*(A):=\int_A m_n(x)^2\mu(dx).
\end{align*}
Since $\mathbb{E}(Y^2) < \infty$ we have $\int m(x)^2 \mu(dx )<\infty$ and hence we also have $\int m_n(x)^2 \mu(dx)  \leq \int m(x)^2 \mu(dx) + o(1) < \infty$ so that $\mu^*$ and $\mu_n^*$ are bounded measures. Thus, a very similar argument to that used to prove \eqref{11} shows that
\begin{equation}
\label{Eq:FourthMoments}
	\lim_{n \rightarrow \infty} \sum_j \frac{\tau_n(A_{h,j})^4}{c_n^4 h_n^{4d}} \mu_n^*(A_{h,j}) = 0 \quad \text{a.s.}
\end{equation}
where we use the fact that $\{\epsilon_{i,j}\}$, $\{\zeta_{i,j}\}$, $\{(X_i,Y_i)\}$
are independent. Then the Cauchy-Schwarz inequality, (\ref{c1}) and (\ref{Eq:FourthMoments}) imply
\begin{align*}
F_n
&=
\int \frac{\tau_n(A_n(x))^2}{\tilde{\mu}_n(A_n(x))^2}\IND_{\{\tilde{\mu}_n(A_n(x))\ge c_nh_n^d\}}\mu^*_n(dx )\\
&=
\sum_j \frac{\tau_n(A_{h_n,j})^2}{\tilde{\mu}_n(A_{h_n,j})^2}\IND_{\{\tilde{\mu}_n(A_{h_n,j})\ge c_nh_n^d\}}\mu^*_n(A_{h_n,j} )\\
&\le
\sum_j \frac{\tau_n(A_{h_n,j})^2}{c_n^2h_n^{2d}}\mu^*_n(A_{h_n,j} )\\
&\le
\sqrt{\sum_j \frac{\tau_n(A_{h_n,j})^4}{c^4_nh_n^{4d}}\mu^*_n(A_{h_n,j} )}\sqrt{\int m_n(x)^2 \mu(dx ) }\\
&\to 0\qquad \mbox{a.s.}
\end{align*}
The fact that $G_n\to 0$ a.s. follows from (\ref{c1}). We now turn to $H_n$. Since we have $\int m(x)^2 \mu(dx )<\infty$, it suffices to show that
\begin{align*}
\int \IND_{\{\tilde{\mu}_n(A_n(x))< c_nh_n^d\}}\mu(dx )
&\to 0\qquad \mbox{a.s.,}
\end{align*}
i.e.,
\begin{align*}
\sum_j \IND_{\{\tilde{\mu}_n(A_{h_n,j})< c_nh_n^d\}}\mu(A_{h_n,j} )
&\to 0\qquad \mbox{a.s.}
\end{align*}
By the inequality
\begin{align*}
&\IND_{\{\tilde{\mu}_n(A_{h_n,j})< c_nh_n^d\}}\\
& \leq
\IND_{\{|\tau_n(A_{h_n,j})| \geq c_nh_n^d/2\}}+ \IND_{\{|\mu_n(A_{h_n,j}) - \mu(A_{h_n,j})| \geq c_nh_n^d/2\}} + \IND_{\{\mu(A_{h,j}) <  2c_nh_n^d\}},
\end{align*}
one gets
\begin{align*}
H_n
& \leq
8\sum_j \frac{\tau_n(A_{h_n,j})^2}{c^2_nh_n^{2d}} \mu(A_{h_n,j})\\
&\quad+
\frac{8}{c_n^2 h_n^{2d}} \sum_j (\mu_n(A_{h_n,j})-\mu(A_{h_n,j}))^2 \mu(A_{h_n,j}) \\
&\quad+ 2\sum_j \IND_{\{\mu(A_{h,j}) <  2c_nh_n^d\}} \mu(A_{h_n,j}).
\end{align*}
(\ref{11}) implies that the first term tends to $0$ a.s.
Concerning the second term, we observe that, by the fact that Bernoulli random variables are subgaussian with variance proxy bounded by $1/4$, there exists $L>0$ such that for any $q \in \mathbb{N}$ we have
\begin{align*}
	\mathbb{E} \bigl[ \bigl( \mu_n(A_{h_n,j}) - \mu(A_{h_n,j}) \bigr)^{2q} \bigr] \leq n^{-q} (L q^{1/2})^{2q}.
\end{align*}
Thus, the second term tends to zero a.s. by using a very similar argument to that used to prove~\eqref{11}.
Finally, the third term is non-random. Let $S$ be a sphere centred at the origin such that $\mu(S^c) \leq \varepsilon$, and set
\[
	B_n := \bigcup_{j: \mu(A_{h_n,j}) < 2c_nh_n^d, A_{h_n,j} \cap S \neq \emptyset} A_{h_n,j}.
\]
If $\lambda$ denotes the Lebesgue measure, then
\begin{align*}
\sum_j \IND_{\{\mu(A_{h,j}) <  2c_nh_n^d\}} \mu(A_{h_n,j})
&\leq \mu(B_n) + \mu(S^c)
\leq \mu(B_n)+\varepsilon
\end{align*}
and
\begin{align*}
\mu(B_n)
& \leq  \sum_{j: \mu(A_{h_n,j}) < 2c_nh_n^d, A_{h_n,j} \cap S \neq \emptyset} 2c_n\lambda(A_{h_n,j}) \\
& \leq 2c_n \sum_{j: A_{h_n,j} \cap S \neq \emptyset} \lambda(A_{h_n,j}) \\
& \rightarrow 0.
\end{align*}
Thus, we proved that $H_n\to 0$ a.s.
\end{proof}

\begin{proof}[Proof of Theorem~\ref{patt}]
For the notation
\begin{align*}
	\bar{m}_n(x) = \frac{\tilde{\nu}_n(A_{h_n,j})}{\mu(A_{h_n,j})} \quad \text{when } x \in A_{h_n,j},
\end{align*}
the rule $g_n$ has the equivalent form
\begin{align*}
g_n(x) = sign\, \bar{m}_n(x).
\end{align*}
Theorem 2.2 in \citet{devroye2013probabilistic} implies that
\begin{align*}
L(g_n)-L^*
&=
\int\IND_{\{g_n(x)\ne g^*(x)\}}|m(x)|\mu(dx)\\
&=
\int\IND_{\{sign\,\bar{m}_n(x) \ne sign\, m(x) \}}|m(x)|\mu(dx)\\
&\le
\int [|m(x)-\bar{m}_n(x)|]_0^1\mu(dx).
\end{align*}
Write
\begin{align*}
m_n(x) = \frac{\nu_n(A_{h_n,j})}{\mu(A_{h_n,j})} \quad \text{when } x \in A_{h_n,j}.
\end{align*}
Then,
\begin{align*}
&\int [|m(x)-\bar{m}_n(x)|]_0^1\mu(dx)\\
&\le
\int |m(x)-m_n(x)|\mu(dx)
+
\int [|m_n(x)-\bar{m}_n(x)|]_0^1\mu(dx).
\end{align*}
By Theorem 23.1 in \citet{gyorfi2006distribution}, the first term tends to $0$ a.s.
Similarly to the previous proof, given $\epsilon>0$ let $S$ be a sphere centred at the origin such that $\mu(S^c) \leq \varepsilon$, and set
\[
	B_n := \bigcup_{j: A_{h_n,j} \cap S \neq \emptyset} A_{h_n,j}.
\]
Then,
\begin{align*}
\int [|m_n(x)-\bar{m}_n(x)|]_0^1\mu(dx)
&\le
\sum_{j\in B_n}[|\nu_n(A_{h_n,j})-\tilde{\nu}_n(A_{h_n,j})|]_0^1
+\mu(S^c)\\
&\le
\sum_{j\in B_n}\left[\left| \frac{\sigma_Z}{n} \sum_{i=1}^n  \epsilon_{i,j}\right|\right]_0^1
+\varepsilon\\
&\le
\sum_{j\in B_n}\left(\varepsilon h_n^d + \IND_{\left[\left| \frac{\sigma_Z}{n} \sum_{i=1}^n  \epsilon_{i,j}\right|\right]_0^1\ge \varepsilon h_n^d}\right)
+\varepsilon.
\end{align*}
For $n$ sufficiently large,
\begin{align*}
\sum_{j\in B_n}\varepsilon h_n^d
&\le 2\lambda (S) \varepsilon
\end{align*}
and Lemma \ref{A} implies
\begin{align*}
\sum_n\EXP\left\{
\sum_{j\in B_n} \IND_{\left| \frac{\sigma_Z}{n} \sum_{i=1}^n  \epsilon_{i,j}\right|\ge \varepsilon h_n^d}\right\}
&=
\sum_n|B_n| \PROB\left\{\left| \frac{\sigma_Z}{n} \sum_{i=1}^n  \epsilon_{i,1}\right|\ge \varepsilon h_n^d\right\}\\
&=
\sum_n\frac{4\lambda (S) }{h_n^d } e^{-n(\varepsilon h_n^d/\sigma_Z)^2/4}\\
&< \infty,
\end{align*}
where the last step follows from the condition $nh_n^{2d}/\log n\to \infty$. Therefore, by Markov's inequality and the Borel-Cantelli lemma, we have proved that
\begin{align*}
\limsup_n \int [|m_n(x)-\bar{m}_n(x)|]_0^1\mu(dx)
&\le
2\lambda (S) \varepsilon+\varepsilon
\end{align*}
a.s. Since $\epsilon>0$ was arbitrary, this completes the proof.
\end{proof}

\bibliography{bib}








\end{document}